\documentclass{amsart}
\usepackage{amsmath,amssymb}
\usepackage{bm}
\usepackage[pdftex]{graphicx}
\usepackage{float}
\usepackage{ascmac} 
\usepackage{fancybox}
\usepackage{tabularx}
\usepackage{multicol}

\newtheorem{dfn}{Definition}[section]
\newtheorem{thm}[dfn]{Theorem}
\newtheorem{prop}[dfn]{Proposition}
\newtheorem{lem}[dfn]{Lemma}

\newtheorem{remark}[dfn]{Remark}

\makeatletter
\renewcommand{\p@enumi}{A.}
\makeatother

\makeatletter
\@addtoreset{equation}{section}

\makeatother

\DeclareMathOperator*{\diam}{diam}
\DeclareMathOperator*{\dist}{dist}

\title[]{Observability inequalities for degenerate transport equations}
\author{Giuseppe Floridia}
\address{Università Mediterranea di Reggio Calabria,
Department PAU
Via dell'Università 25  
89124 Reggio Calabria, Italy
\&
INdAM Unit, Univ. di Catania, Italy}
\email{floridia.giuseppe@icloud.com}

\author{Hiroshi Takase}
\address{Graduate School of Mathematical Sciences, The University of Tokyo, 3-8-1 Komaba, Meguro-ku, Tokyo 153-8914, Japan}
\email{htakase@ms.u-tokyo.ac.jp}
\date{July 25, 2021}
\begin{document}
\begin{abstract}
In this paper we prove an observability inequality for a degenerate transport equation with time-dependent coefficients. First we introduce a local in time Carleman estimate for the degenerate equation, then we apply it to obtain a global in time observability inequality by using also an energy estimate.
\end{abstract}
\maketitle
\section{Introduction and Main result}
Let $d\in\mathbb{N}$, $T>0$, and $\Omega\subset\mathbb{R}^d$ be a bounded domain with smooth boundary $\partial\Omega$ and $\nu(x)$ be the unit outer normal to $\partial\Omega$ at $x\in\partial\Omega$. Without loss of generality, we suppose $0\in\Omega$. We set $Q:=\Omega\times(0,T)$ and $\Sigma:=\partial\Omega\times(0,T)$. We introduce the differential operator $A$ such that
\begin{equation}\label{A}Au(x,t):=\partial_tu+H(t)\cdot\nabla u,\end{equation}
where $H(t):=(H_1(t),\ldots,H_d(t))$ is a continuous vector-valued function on $[0,T]$.

A lot of inverse problems via Carleman estimates for transport equations have been studied. Klibanov and Pamyatnykh \cite{Klibanov2008} proved a global uniqueness theorem for an inverse coefficient problem. Gaitan and Ouzzane \cite{Gaitan2014}, Machida and Yamamoto \cite{Machida2014}, and G\"olgeleyen and Yamamoto \cite{Golgeleyen2016} proved Lipschitz stabilities for inverse coefficient and source problems via global Carleman estimates for transport equations with variable coefficients. Cannarsa, Floridia, G\"olgeleyen, and Yamamoto \cite{Cannarsa2019} proved local H\"older stability to determine principal terms and zeroth-order terms. We should note that these results were all for transport equations the coefficients of which do not depend on time variable $t$ but depend on space variable $x$. In regard to transport equations having a time-dependent principal part, Cannarsa, Floridia, and Yamamoto \cite{Cannarsa2019a} proved an observability inequality for the operator $A$ defined by \eqref{A} with $|H(t)|>0$ for all $t\in[0,T]$, i.e., non-degenerate case, which was motivated by applications to inverse problems. In this paper, we eliminate the assumption on the positivity of $|H(t)|$ and prove the observability inequality in the degenerate case. Although this paper is inspired by \cite{Cannarsa2019a}, we note that our methodology is a little different from it, since we do not use the partition arguments employed in \cite{Cannarsa2019a}. Moreover, we prove the observability inequality using a synthetic technique recently introduced in \cite{Huang2020} by Huang, Imanuvilov, and Yamamoto, without using the classical cut-off arguments in the proof of the observability through the Carleman estimate. This enables us to simplify proofs of observability inequalities. 

For more applications of Carleman estimates to inverse problems, controllability, and unique continuations for hyperbolic equations, readers are referred to Bellassoued and Yamamoto \cite{Yamamoto2017}, and Takase \cite{Takase2020a}. They established Carleman estimates for second-order hyperbolic operators with variable coefficients on manifolds.
Moreover, for degenerate evolutions equations there is a extensive literature, one can see, e.g., Floridia \cite{Floridia2014} and Floridia, Nitsch, and Trombetti \cite{Floridia2020}.

\subsubsection*{Structure of the paper} In this section, after describing the problem formulation and some notations, we present our main result in Theorem \ref{observability}.
In section 2, we prepare some propositions needed to prove Theorem \ref{observability}. In particular, we obtain the energy estimate (see Lemma \ref{energy} and Proposition \ref{continuity}), and the Carleman estimate for the degenerate case (see Proposition \ref{Carleman}), which play important roles in proving the main result. Finally, in section 3 we prove Theorem \ref{observability}. In section 4, using the methodology of this paper we obtain an observability inequality for the non-degenerate case, studied in \cite{Cannarsa2019a}, by a proof shorter than one in \cite{Cannarsa2019a}.

In this paper, we consider the degenerate case, where we impose the following assumptions on the vector field $H\in C([0,T];\mathbb{R}^d)$:
\begin{equation}\label{degenerate}H(0)=0;\end{equation}
\begin{gather}\label{positivity}\exists T_1\in(0,T],\ \exists\rho>0\ \text{s.t.}\ H\in C^1([0,T_1];\mathbb{R}^d)\ \text{and}\\
\min_{t\in[0,T_1]}|H'(t)|\ge\rho.\notag\end{gather}
Under assumptions \eqref{degenerate} and \eqref{positivity}, we consider the Cauchy problem
\begin{align}\label{boundary}\begin{cases}Au=\partial_tu+H(t)\cdot\nabla u=0\quad &\text{in}\ Q,\\
u=g\quad &\text{on}\ \Sigma,\end{cases}\end{align}
where $g\in L^2(\Sigma)$, and prove an observability inequality in Theorem \ref{observability}. Unlike the non-degenerate case by Cannarsa, Floridia, and Yamamoto \cite{Cannarsa2019}, we should impose the extra assumption \eqref{positivity} on the positivity of $|H'(t)|$ due to the degeneration \eqref{degenerate}. Nevertheless, the regularity class in \eqref{positivity} imposed on $H$ is the same one as in \cite{Cannarsa2019a}.
 
Before describing mathematical settings, we mention a synthetic statement of the main result Theorem \ref{observability}. We note to prove an observability inequality for a hyperbolic equation the observation time should be given sufficiently large due to the finite propagation speed (e.g., Bardos, Lebeau, and Rauch \cite{Bardos1992}). Theorem \ref{observability} claims that if the direction of the unit vector $\frac{H'(t)}{|H'(t)|}$ changes moderately comparing with the time for the distant wave to reach the boundary, then we can obtain the observability inequality \eqref{observ_ineq} for a sufficient large observation time. To formulate this situation mathematically, we define some preliminary notations.

\begin{dfn}\label{t_1}
Let $T>0$, $c_0\in(\frac{1}{\sqrt{2}},1)$, and $H$ be a vector-valued function satisfying \eqref{positivity}. We define a positive number $t_1\in(0,T_1]$ such that
\begin{equation}\label{def}t_1:=\sup\left\{\tau\in[0,T_1]\ \middle|\ \frac{H'(t)}{|H'(t)|}\cdot\frac{H'(0)}{|H'(0)|}\ge c_0,\quad \forall t\in[0,\tau]\right\}.\end{equation}
\end{dfn}

\begin{remark}
Note that $t_1>0$ because $H'$ is continuous.
\end{remark}


By the definition of the positive time $t_1\in(0,T_1]$ introduced in \eqref{def}, the angle between $\frac{H'(t)}{|H'(t)|}$ and $\frac{H'(0)}{|H'(0)|}$ is less than or equal to $\frac{\pi}{4}$ for $t\in [0,t_1]$. The positive time $t_1$ will be crucial to prove the observability inequality \eqref{observ_ineq} in Theorem \ref{observability}. The next lemma is a basic property for $H'$ in the time interval $[0,t_1]$.

\begin{lem}\label{angle_lem}
Let $T>0$, $c_0\in(\frac{1}{\sqrt{2}},1)$, $H$ be a vector-valued function satisfying \eqref{positivity}, and $t_1\in(0,T_1]$ be the positive number defined by \eqref{def}. Then, there exists $x_0\in\overline{\Omega}^c:=\mathbb{R}^d\setminus\overline{\Omega}$ such that
\begin{equation}\label{angle}\min_{(x,t)\in\overline{\Omega}\times[0,t_1]}\frac{H'(t)\cdot(x-x_0)}{|H'(t)||x-x_0|}\ge 2c_0^2-1(>0).\end{equation}
\end{lem}
\begin{proof}
If we take $x_0:=-R\theta_0\in\overline{\Omega}^c$ for $R>\frac{1+c_0}{1-c_0}\diam\Omega$ and $\theta_0:=\frac{H'(0)}{|H'(0)|}$, we find
\begin{align*}(x-x_0)\cdot\theta_0&=x\cdot\theta_0+R\ge R-|x|\ge R-\diam\Omega\\
&>c_0(R+\diam\Omega)\ge c_0|x-x_0|\end{align*}
holds for all $x\in\overline{\Omega}$, which implies $\displaystyle\min_{(x,t)\in\overline{\Omega}\times[0,t_1]}\frac{x-x_0}{|x-x_0|}\cdot\theta_0\ge c_0$. Moreover, taking $\displaystyle\min_{(x,t)\in\overline{\Omega}\times[0,t_1]}\frac{H'(t)}{|H'(t)|}\cdot\theta_0\ge c_0$ into account, we finally conclude \eqref{angle} is true by the trigonometric addition formulas for the angle between $\frac{x-x_0}{|x-x_0|}$ and $\theta_0$, and the angle between $\frac{H'(t)}{|H'(t)|}$ and $\theta_0$.
\end{proof}

\begin{figure}[htbp]
\centering\includegraphics[scale=0.25]{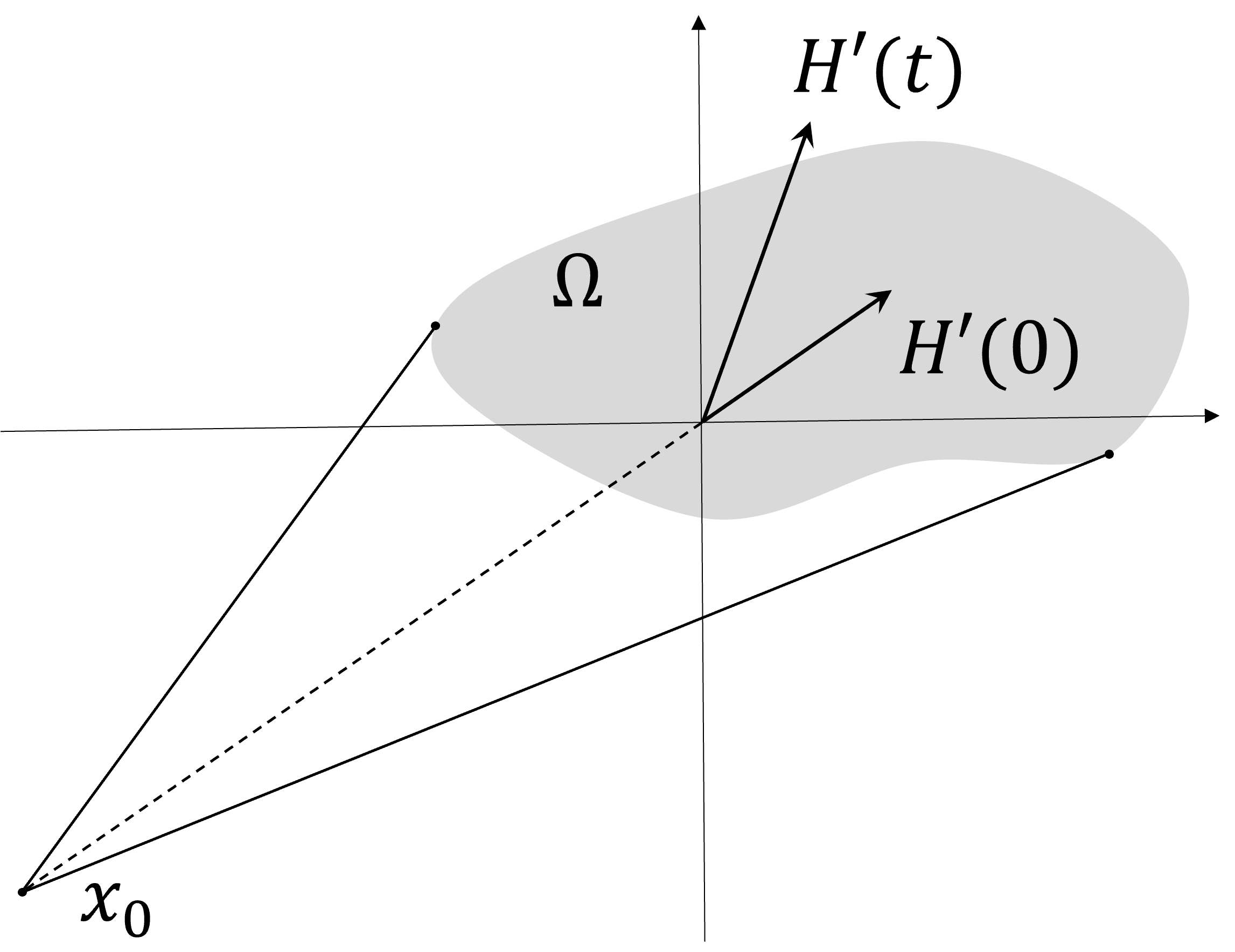}
\caption{The situation of $H'(t)$ and $x_0\in\overline{\Omega}^c$ in Lemma \ref{angle_lem}.}
\end{figure}

For a fixed $x_0\in\overline{\Omega}^c$ satisfying \eqref{angle}, define the positive number
\begin{equation}\label{T_0}T_0:=\sqrt{\frac{\displaystyle\max_{x\in\overline{\Omega}}|x-x_0|^2-\min_{x\in\overline{\Omega}}|x-x_0|^2}{\delta}},\end{equation}
where
\begin{equation}\label{delta}\delta:=\rho(2c_0^2-1)\dist(x_0,\Omega)>0.\end{equation}

The next theorem is our main result in this paper.

\begin{thm}\label{observability}
Let $T>0$, $c_0\in(\frac{1}{\sqrt{2}},1)$, $H\in C([0,T];\mathbb{R}^d)$, and $g\in L^2(\Sigma)$. Assume \eqref{degenerate} and \eqref{positivity}. If the number $t_1\in(0,T_1]$ defined by \eqref{def} satisfies $T_0<t_1$ for some $x_0\in\overline{\Omega}^c$ satisfying \eqref{angle}, then there exists a constant $C>0$ independent of $g\in L^2(\Sigma)$ such that for all $t\in[0,T]$,
\begin{equation}\label{observ_ineq}\|u(\cdot,t)\|_{L^2(\Omega)}\le C\|g\|_{L^2(\Sigma)}\end{equation}
holds for all $u\in H^1(Q)$ satisfying \eqref{boundary}.
\end{thm}

\section{Preliminaries}
In this section, we prepare some results needed to prove Theorem \ref{observability}. In section 2.1, by the energy estimate Lemma \ref{energy} we prove Proposition \ref{continuity}, which means if the observability inequality \eqref{observ_ineq} holds locally in time, then it holds also globally in time. In section 2.2, we present the Carleman estimate in Proposition \ref{Carleman}.

\subsection{Energy estimate}
For the proof of Theorem \ref{observability}, we use the energy estimate of the following type, which is proved without assuming \eqref{degenerate} and \eqref{positivity}.
\begin{lem}\label{energy}
Let  $T>0$, $H\in C([0,T];\mathbb{R}^d)$, and $g\in L^2(\Sigma)$. Then, there exists a constant $C>0$ independent of $g\in L^2(\Sigma)$ such that for all $t\in[0,T]$,
\[\left|\|u(\cdot,t)\|_{L^2(\Omega)}^2-\|u(\cdot,0)\|_{L^2(\Omega)}^2\right|\le C\|g\|_{L^2(\Sigma)}^2\]
holds for all $u\in H^1(Q)$ satisfying \eqref{boundary}.
\end{lem}
\begin{proof}
Multiplying the equation in \eqref{boundary} by $2u$ and integrating over $\Omega$ yield
\[\int_\Omega\partial_t(|u|^2)dx+\int_\Omega H(t)\cdot\nabla (|u|^2)dx=0,\]
i.e.,
\[\frac{d}{dt}\left(\int_\Omega|u|^2dx\right)=-\int_{\partial\Omega}\left(H(t)\cdot\nu(x)\right)|g|^2d\sigma.\]
Integration over $[0,t]$ yields
\[\left|\|u(\cdot,t)\|_{L^2(\Omega)}^2-\|u(\cdot,0)\|_{L^2(\Omega)}^2\right|\le C\|g\|_{L^2(\Sigma)}^2,\]
for some $C>0$ independent of $g\in L^2(\Sigma)$, $t\in[0,T]$, and $u\in H^1(Q)$.
\end{proof}

\begin{prop}\label{continuity}
Let  $T>0$, $H\in C([0,T];\mathbb{R}^d)$, and $g\in L^2(\Sigma)$. Assume there exist $\tau\in[0,T]$ and a constant $C_1>0$ independent of $g\in L^2(\Sigma)$ such that for all $t\in[0,\tau]$,
\[\|u(\cdot,t)\|_{L^2(\Omega)}\le C_1\|g\|_{L^2(\Sigma)}\]
holds for all $u\in H^1(Q)$ satisfying \eqref{boundary}. Then, there exists a constant $C_2>0$ independent of $g\in L^2(\Sigma)$ such that 
\[\|u(\cdot,t)\|_{L^2(\Omega)}\le C_2\|g\|_{L^2(\Sigma)}\]
holds for all $t\in[0,T]$ and $u\in H^1(Q)$ satisfying \eqref{boundary}.
\end{prop}
\begin{proof}
The claim is trivial when $\tau=T$. When $\tau<T$, Lemma \ref{energy} and the assumption in Proposition \ref{continuity} yield
\begin{align*}\|u(\cdot,t)\|_{L^2(\Omega)}^2&\le \|u(\cdot,0)\|_{L^2(\Omega)}^2+C\|g\|_{L^2(\Sigma)}^2\\
&\le (C_1^2+C)\|g\|_{L^2(\Sigma)}^2\end{align*}
for all $t\in[0,T]$ and $u\in H^1(Q)$ satisfying \eqref{boundary}. If we set $C_2:=\sqrt{C_1^2+C}$, we complete the proof.
\end{proof}

\subsection{Carleman estimate}
Let $\tau>0$ and $c_0\in(\frac{1}{\sqrt{2}},1)$ be fixed constants. We set $Q_{\pm,\tau}:=\Omega\times(-\tau,\tau)$ and $\Sigma_{\pm,\tau}:=\partial\Omega\times(-\tau,\tau)$. In this section, we establish the Carleman estimate for the differential operator $A$,
\[Au:=\partial_tu+H(t)\cdot\nabla u\]
in $Q_{\pm,\tau}$ under the following assumptions:
\begin{equation}\label{A.1}H\in C^1([-\tau,\tau];\mathbb{R}^d);\end{equation}
\begin{equation}\label{A.2}\exists\rho>0\ \text{s.t.}\ \min_{t\in[-\tau,\tau]}|H'(t)|\ge\rho;\end{equation}
\begin{equation}\label{A.3}\exists\theta_0\in \mathbb{S}^{d-1}\ \text{s.t.}\ \min_{t\in[-\tau,\tau]}\frac{H'(t)\cdot\theta_0}{|H'(t)|}\ge c_0,\end{equation}
where $\mathbb{S}^{d-1}:=\{\xi\in\mathbb{R}^d\mid|\xi|=1\}$.

Under the assumptions \eqref{A.1}--\eqref{A.3}, we will obtain the Carleman estimate for $A$ in $Q_{\pm,\tau}$.

We can take $x_0\in\overline{\Omega}^c:=\mathbb{R}^d\setminus\overline{\Omega}$ satisfying $\displaystyle\min_{(x,t)\in\overline{Q_{\pm,\tau}}}\frac{H'(t)\cdot(x-x_0)}{|H'(t)||x-x_0|}\ge 2c_0^2-1$ by the same argument as in the proof of Lemma \ref{angle_lem}.

For a positive constant $\beta>0$ to be fixed later, we set
\begin{equation}\label{weight}\varphi(x,t):=|x-x_0|^2-\beta t^2,\quad (x,t)\in \overline{Q_{\pm,\tau}}.\end{equation}
We establish the Carleman estimate Proposition \ref{Carleman} for the operator $A$ having time-dependent coefficients. Nevertheless, our choice of weight functions is more similar to the one by Klibanov--Pamyatnykh \cite{Klibanov2008} and Gaitan--Ouzzane \cite{Gaitan2014} than by Cannarsa--Floridia--Yamamoto \cite{Cannarsa2019}.

\begin{prop}\label{Carleman}
Assume \eqref{A.1}, \eqref{A.2}, and \eqref{A.3}. Let $\varphi$ be the smooth function defined by \eqref{weight}, where $\beta>0$ is an arbitrary positive number satisfying
\[0<\beta<\delta:=\rho(2c_0^2-1)\dist(x_0,\Omega).\]
Then, there exists a constant $C>0$ such that
\begin{align}\label{estimate}
&s\int_{Q_{\pm,\tau}} e^{2s\varphi}|u|^2dxdt\\
&\quad\le C\int_{Q_{\pm,\tau}} e^{2s\varphi}|Au|^2dxdt+Cs\int_{\Sigma_{\pm,\tau}} e^{2s\varphi}A\varphi(H(t)\cdot \nu(x))|u|^2d\sigma dt\notag\\
&\quad\quad+Cs\int_\Omega \left(e^{2s\varphi(x,\tau)}|u(x,\tau)|^2+e^{2s\varphi(x,-\tau)}|u(x,-\tau)|^2\right)dx\notag\end{align}
holds for all $s>0$ and $u\in H^1(Q_{\pm,\tau})$. Here $d\sigma$ denotes the volume element of $\partial\Omega$.
\end{prop}

\begin{proof}
Set $z:=e^{s\varphi}u$ and $Pz:=e^{s\varphi}A(e^{-s\varphi}z)$ for $u\in H^1(Q_{\pm,\tau})$ and $s>0$. Since $\varphi$ is smooth, it follows that $z\in H^1(Q_{\pm,\tau})$. We note that
\[A\varphi(x,t)=-2\beta t+2H(t)\cdot (x-x_0)\]
and
\[A^2\varphi=-2\beta+2H'(t)\cdot (x-x_0)+2|H(t)|^2.\]
Since we have
\[Pz=Az-s(A\varphi)z,\]
\begin{align}\label{inner_prod}&\|Pz\|_{L^2(Q_{\pm,\tau})}^2\ge 2(Az,-s(A\varphi)z)_{L^2(Q_{\pm,\tau})}\\
&\quad=-2s\int_{Q_{\pm,\tau}}(\partial_tz+H(t)\cdot\nabla z)(A\varphi)zdxdt\notag\\
&\quad=-s\int_{Q_{\pm,\tau}}(A\varphi)\partial_t(|z|^2)dxdt-s\int_{Q_{\pm,\tau}}(A\varphi)H(t)\cdot\nabla(|z|^2)dxdt\notag\\
&\quad=s\int_{Q_{\pm,\tau}} A^2\varphi |z|^2dxdt-s\int_{\Sigma_{\pm,\tau}} A\varphi(H(t)\cdot\nu(x))|z|^2d\sigma dt\notag\\
&\quad\quad-s\int_\Omega\Big[A\varphi|z|^2\Big]_{t=-\tau}^{t=\tau}dx\notag\\
&\quad\ge 2s\int_{Q_{\pm,\tau}}\Big(-\beta+H'(t)\cdot (x-x_0)\Big)|z|^2dxdt\notag\\
&\quad\quad-s\int_{\Sigma_{\pm,\tau}} (A\varphi)(H(t)\cdot\nu(x))|z|^2d\sigma dt-s\int_\Omega\Big[A\varphi|z|^2\Big]_{t=-\tau}^{t=\tau}dx\notag\end{align}
holds. For the fixed $x_0\in\overline{\Omega}^c$ so that $\displaystyle\min_{(x,t)\in\overline{Q_{\pm,\tau}}}\frac{H'(t)\cdot(x-x_0)}{|H'(t)||x-x_0|}\ge 2c_0^2-1(>0)$, it follows that
\begin{align*}H'(t)\cdot(x-x_0)&=|H'(t)||x-x_0|\frac{H'(t)\cdot (x-x_0)}{|H'(t)||x-x_0|}\\
&\ge\rho\dist(x_0,\Omega)\min_{(x,t)\in\overline{Q_{\pm,\tau}}}\frac{H'(t)\cdot (x-x_0)}{|H'(t)||x-x_0|}\\
&\ge\delta(>0)\end{align*}
for all $(x,t)\in \overline{Q_{\pm,\tau}}$ owing to \eqref{A.2} and \eqref{A.3}. We then obtain from \eqref{inner_prod}
\begin{align*}\|Pz\|_{L^2(Q_{\pm,\tau})}^2&\ge 2(\delta-\beta)s\int_{Q_{\pm,\tau}}|z|^2dxdt-s\int_{\Sigma_{\pm,\tau}} (A\varphi)(H(t)\cdot\nu(x))|z|^2d\sigma dt\\
&\quad-s\int_\Omega\Big[A\varphi|z|^2\Big]_{t=-\tau}^{t=\tau}dx.\end{align*}
Hence, for all $0<\beta<\delta$, there exists a constant $C>0$ such that
\begin{align*}&s\int_{Q_{\pm,\tau}} e^{2s\varphi}|u|^2dxdt\\
&\quad\le C\int_{Q_{\pm,\tau}} e^{2s\varphi}|Au|^2dxdt+Cs\int_{\Sigma_{\pm,\tau}} e^{2s\varphi}A\varphi(H(t)\cdot \nu(x))|u|^2d\sigma dt\\
&\quad\quad+Cs\int_\Omega \left(e^{2s\varphi(x,\tau)}|u(x,\tau)|^2+e^{2s\varphi(x,-\tau)}|u(x,-\tau)|^2\right)dx\end{align*}
holds for all $s>0$ and $u\in H^1(Q_{\pm,\tau})$.
\end{proof}

\begin{remark}
In Proposition \ref{Carleman}, we do not assume the positivity of $|H(t)|$. In that respect, Proposition \ref{Carleman} is different from Theorem 1.5 in Cannarsa--Floridia--Yamamoto \cite{Cannarsa2019}. Proposition \ref{Carleman} says the Carleman estimate holds regardless of whatever $|H(t)|$ is positive if we assume appropriate properties in regard to $H'$.

The technical difference appears in the estimate \eqref{inner_prod}. In the non-degenerate case (e.g., \cite{Cannarsa2019} and Proposition \ref{Carleman'} in this paper), we can use the positivity of $A\varphi$. However, in the degenerate case, we use the positivity of $A^2\varphi$.

\end{remark}

\section{Proof of Theorem \ref{observability}}
To prove the main result, we use not only Lemma \ref{energy} and Proposition $\ref{continuity}$ but also Proposition \ref{Carleman}, i.e., the Carleman estimate for the operator $A$. Furthermore, we should describe a technical remark in applying Carleman estimates. In existing works, whenever we applied Carleman estimates to obtain stability estimates for some inverse problems, we introduced appropriate cut-off functions $\chi$ and applied Carleman estimates to $\chi u$, where $u$ is a solution to considering equations. This was because $\chi u$ vanished on boundaries of considering domains. However, in our proof of Theorem \ref{observability}, we need not use the cut-off arguments because our Carleman estimate in Proposition \ref{Carleman} contains all the boundary terms on $\partial Q_{\pm,\tau}$. This argument without cut-off functions is presented by Huang, Imanuvilov, and Yamamoto \cite{Huang2020}.

\begin{proof}[Proof of Theorem \ref{observability}]
In the beginning, we extend $H\in C([0,T];\mathbb{R}^d)$ and $u\in H^1(Q)$ satisfying \eqref{boundary} in $Q_\pm:=\Omega\times(-T,T)$ by setting
\[\bar{H}(t)=\begin{cases}H(t),\quad &t\in[0,T],\\
-H(-t),\quad &t\in[-T,0],\end{cases}\]
and
\[u(x,t)=\begin{cases}u(x,t)\quad &\text{in}\ \Omega\times(0,T),\\
u(x,-t)\quad &\text{in}\ \Omega\times(-T,0).\end{cases}\]
By our assumptions \eqref{degenerate} and \eqref{positivity}, $\bar{H}\in C([-T,T];\mathbb{R}^d)\cap C^1([-T_1,T_1];\mathbb{R}^d)$ and $u\in H^1(Q_\pm)$. Furthermore, the derivatives with respect to $t$ of $\bar{H}$ and $u$ satisfy
\[\bar{H}'(t)=\begin{cases}H'(t),\quad &t\in[0,T_1],\\
H'(-t),\quad &t\in[-T_1,0],\end{cases}\]
and
\[\partial_tu(x,t)=\begin{cases}\partial_tu(x,t)\quad &\text{in}\ \Omega\times(0,T),\\
-\partial_tu(x,-t)\quad &\text{in}\ \Omega\times(-T,0),\end{cases}\]
which imply $u$ satisfies
\begin{align}\label{boundary2}\begin{cases}Au=\partial_tu+\bar{H}(t)\cdot\nabla u=0\quad &\text{in}\ Q_\pm,\\
u=\bar{g}\quad &\text{on}\ \Sigma_\pm:=\partial\Omega\times(-T,T),\end{cases}\end{align}
where $\bar{g}$ is extended by
\begin{equation}\label{g}\bar{g}(x,t)=\begin{cases}g(x,t)\quad &\text{in}\ \partial\Omega\times(0,T),\\
g(x,-t)\quad &\text{in}\ \partial\Omega\times(-T,0).\end{cases}\end{equation}
Let $t_1>0$ be the positive number defined by \eqref{def} and $x_0\in\overline{\Omega}^c$ be the point satisfying \eqref{angle} under the assumption
\[T_0<t_1,\]
where $T_0$ is defined by \eqref{T_0}. Owing to Proposition \ref{continuity}, it suffices to prove the observability inequality \eqref{observ_ineq} in the interval $[0,t_1],$ then we can extend it to all the interval $[0,T]$.

For the fixed $x_0\in\overline{\Omega}^c$, we take $0<\beta<\delta$, where $\delta$ is defined by \eqref{delta}, satisfying
\[(T_0<)\sqrt{\frac{d_M-d_m}{\beta}}<t_1,\]
where we define
\[d_M:=\max_{x\in\overline{\Omega}}|x-x_0|^2,\quad d_m:=\min_{x\in\overline{\Omega}}|x-x_0|^2.\]
Then, there exists $\kappa>0$ such that
\begin{equation}\label{negative}d_M-d_m-\beta t_1^2<-\kappa.\end{equation}
Henceforth, by $C>0$ we denote a generic constant independent of $u$ and $\bar{g}$ which may change from line to line, unless specified otherwise. We find that $\bar{H}$ satisfies the assumptions \eqref{A.1}--\eqref{A.3} of section 2.2 by taking $\tau=t_1$ and $\theta_0=\frac{H'(0)}{|H'(0)|}$ needed for Proposition \ref{Carleman}. Set $Q_{\pm,t_1}:=\Omega\times(-t_1,t_1)$ and $\Sigma_{\pm,t_1}:=\partial\Omega\times(-t_1,t_1)$. Applying Proposition \ref{Carleman} to the extended $u\in H^1(Q_{\pm,t_1})$ satisfying \eqref{boundary2} yields
\begin{align}\label{C}&s\int_{Q_{\pm,t_1}} e^{2s\varphi}|u|^2dxdt\\
&\quad\le Cs\int_{\Sigma_{\pm,t_1}} e^{2s\varphi}A\varphi(\bar{H}(t)\cdot \nu(x))|u|^2d\sigma dt\notag\\
&\quad\quad+Cs\int_\Omega \left(e^{2s\varphi(x,t_1)}|u(x,t_1)|^2+e^{2s\varphi(x,-t_1)}|u(x,-t_1)|^2\right)dx.\notag\end{align}
On the left-hand side of \eqref{C}, we obtain
\begin{equation}\label{sharp}s\int_{Q_{\pm,t_1}}e^{2s\varphi}|u|^2dxdt\ge se^{2s(d_m-\beta\epsilon^2)}\int_{-\epsilon}^\epsilon\int_\Omega |u|^2dxdt,\end{equation}
where $\epsilon\in(0,t_1)$ is an arbitrary small constant satisfying for all $x\in\overline{\Omega}$ and $|t|\le\epsilon$,
\[\varphi(x,t)>0,\]
i.e.,
\begin{equation}\label{min}d_m-\beta\epsilon^2> 0.\end{equation}
Furthermore, keeping in mind that $u$ is the even extension, applying Lemma \ref{energy} in \eqref{sharp}, we have
\begin{align}\label{lhs}s\int_{Q_{\pm,t_1}} e^{2s\varphi}|u|^2dxdt&\ge 2se^{2s(d_m-\beta\epsilon^2)}\int_0^\epsilon\int_\Omega |u|^2dxdt\\
&\ge 2\epsilon se^{2s(d_m-\beta\epsilon^2)}\left(\|u(\cdot,0)\|_{L^2(\Omega)}^2-C\|g\|_{L^2(\Sigma)}^2\right).\notag\end{align}
Moreover, in regard to the second summand of the right-hand side of \eqref{C}, applying Lemma \ref{energy} yields
\begin{align}\label{rhs}&Cs\int_\Omega \left(e^{2s\varphi(x,t_1)}|u(x,t_1)|^2+e^{2s\varphi(x,-t_1)}|u(x,-t_1)|^2\right)dx\\
&\quad\le Cse^{2s(d_M-\beta t_1^2)}\left(\|u(\cdot,t_1)\|_{L^2(\Omega)}^2+\|u(\cdot,-t_1)\|_{L^2(\Omega)}^2\right)\notag\\
&\quad\le 2Cse^{2s(d_M-\beta t_1^2)}\left(\|u(\cdot,0)\|_{L^2(\Omega)}^2+C\|g\|_{L^2(\Sigma)}^2\right).\notag\end{align}
From \eqref{C}, \eqref{lhs}, and \eqref{rhs}, keeping in mind \eqref{g}, we obtain
\begin{align*}&2\epsilon s e^{2s(d_m-\beta\epsilon^2)}\left(\|u(\cdot,0)\|_{L^2(\Omega)}^2-C\|g\|_{L^2(\Sigma)}^2\right)\\
&\quad\le Cse^{2s(d_M-\beta t_1^2)}\left(\|u(\cdot,0)\|_{L^2(\Omega)}^2+C\|g\|_{L^2(\Sigma)}^2\right)+Cse^{Cs}\|g\|_{L^2(\Sigma)}^2,\end{align*}
i.e.,
\[e^{2s(d_m-\beta\epsilon^2)}\left(2\epsilon-Ce^{2s(d_M-d_m-\beta t_1^2+\beta\epsilon^2)}\right)\|u(\cdot,0)\|_{L^2(\Omega)}^2\le Ce^{Cs}\|g\|_{L^2(\Sigma)}^2.\]
Applying \eqref{negative} and \eqref{min} to the left-hand side of the above inequality yields
\[\left(2\epsilon-Ce^{-2s(\kappa-\beta\epsilon^2)}\right)\|u(\cdot,0)\|_{L^2(\Omega)}^2\le Ce^{Cs}\|g\|_{L^2(\Sigma)}^2.\]
By choosing $s>0$ large enough to satisfy $2\epsilon-Ce^{-2s(\kappa-\beta\epsilon^2)}>0$ for the sufficiently small $\epsilon>0$ and applying Lemma \ref{energy} for \eqref{boundary2} again on the left-hand side of the above inequality, we have
\[\|u(\cdot,t)\|_{L^2(\Omega)}^2\le C\|g\|_{L^2(\Sigma)}^2\]
for all $t\in[0,t_1]$.
\end{proof}

\begin{remark}
In Theorem \ref{observability}, the degenerate point $t_*\in[0,T]$ on which $H(t_*)=0$ could be not necessarily equal to $0$. Indeed, by similar arguments to Lemma \ref{energy} and Proposition \ref{continuity}, it suffices to prove the observability inequality in a closed time interval containing $t_*$. Therefore, if there exists a sufficiently long time interval containing $t_*$ on which $\frac{H'(t)}{|H'(t)|}\cdot\frac{H'(t_*)}{|H'(t_*)|}\ge c_0$ holds, we can prove the observability inequality on the time interval by the same way as in the proof of Theorem \ref{observability} using the extension.
\end{remark}

\section{Non-degenerate transport equations}
In this section, we prove the observability inequality for the non-degenerate case studied by Cannarsa, Floridia, and Yamamoto \cite{Cannarsa2019a} without the partition arguments and cut-off arguments. Given $T>0$, we replace the assumption \eqref{degenerate} and \eqref{positivity} on $H\in C([0,T];\mathbb{R}^d)$ with the following:

\begin{equation}\label{positivity'}\exists T_1'\in(0,T],\ \exists\rho>0\ \text{s.t.}\ \min_{t\in[0,T_1']}|H(t)|\ge\rho.\end{equation}

\subsection{Preliminaries}
Our methodology is based on the energy estimate given in Proposition \ref{continuity}, which still holds for the non-degenerate case. We define a positive number corresponding to $t_1$ in Definition \ref{t_1}.
\begin{dfn}\label{t_1'}
Let $T>0$, $c_0\in(\frac{1}{\sqrt{2}},1)$, and $H\in C([0,T];\mathbb{R}^d)$ be a vector-valued function satisfying \eqref{positivity'}. We define a positive number $t_1'\in(0,T_1']$ such that
\begin{equation}\label{def'}t_1':=\sup\left\{\tau\in[0,T_1']\ \middle|\ \frac{H(t)}{|H(t)|}\cdot\frac{H(0)}{|H(0)|}\ge c_0,\quad \forall t\in[0,\tau]\right\}.\end{equation}
\end{dfn}

\begin{lem}\label{angle_lem'}
Let $T>0$, $c_0\in(\frac{1}{\sqrt{2}},1)$, $H\in C([0,T];\mathbb{R}^d)$ be a vector-valued function satisfying \eqref{positivity'}, and $t_1'\in(0,T_1']$ be the positive number defined by \eqref{def'}. Then, there exists $x_0\in\overline{\Omega}^c:=\mathbb{R}^d\setminus\overline{\Omega}$ such that
\begin{equation}\label{angle'}\min_{(x,t)\in\overline{\Omega}\times[0,t_1']}\frac{H(t)\cdot(x-x_0)}{|H(t)||x-x_0|}\ge 2c_0^2-1(>0).\end{equation}
\end{lem}
\begin{proof}
If we take $x_0:=-R\theta_0\in\overline{\Omega}^c$ for $R>\frac{1+c_0}{1-c_0}\diam\Omega$ and $\theta_0:=\frac{H(0)}{|H(0)|}$, we find
\begin{align*}(x-x_0)\cdot\theta_0&=x\cdot\theta_0+R\ge R-|x|\ge R-\diam\Omega\\
&>c_0(R+\diam\Omega)\ge c_0|x-x_0|\end{align*}
holds for all $x\in\overline{\Omega}$, which implies $\displaystyle\min_{(x,t)\in\overline{\Omega}\times[0,t_1']}\frac{x-x_0}{|x-x_0|}\cdot\theta_0\ge c_0$. By the same argument as in the proof of Lemma \ref{angle_lem}, we find \eqref{angle'} holds true.
\end{proof}

One of the most important tools in our methodology is the Carleman estimate. Let $\tau>0$ and $c_0\in(\frac{1}{\sqrt{2}},1)$ be constants. We set $Q_\tau:=\Omega\times(0,\tau)$ and $\Sigma_\tau:=\partial\Omega\times(0,\tau)$. We assume \eqref{B.1}--\eqref{B.3} for the non-degenerate case instead of \eqref{A.1}--\eqref{A.3} for the degenerate case:

\begin{equation}\label{B.1} H\in C^1([0,\tau];\mathbb{R}^d);\end{equation}
\begin{equation}\label{B.2} \exists\rho>0\ \text{s.t.}\ \min_{t\in[0,\tau]}|H(t)|\ge\rho;\end{equation}
\begin{equation}\label{B.3} \exists\theta_0\in \mathbb{S}^{d-1}\ \text{s.t.}\ \min_{t\in[0,\tau]}\frac{H(t)\cdot\theta_0}{|H(t)|}\ge c_0.\end{equation}

In the non-degenerate case, we choose a different weight function from \eqref{weight}. For a constant $\beta>0$, let us define
\begin{equation}\label{weight'}\psi(x,t):=|x-x_0|^2-\beta t,\quad (x,t)\in \overline{Q_{\tau}},\end{equation}
where $x_0\in\overline{\Omega}^c$ is a point satisfying $\displaystyle\min_{(x,t)\in\overline{Q_\tau}}\frac{H(t)\cdot (x-x_0)}{|H(t)||x-x_0|}\ge 2c_0^2-1$.

\begin{prop}\label{Carleman'}
Assume \eqref{B.1}, \eqref{B.2}, and \eqref{B.3}. Let $\psi$ be the smooth function defined by \eqref{weight'}, where $\beta>0$ is an arbitrary positive number satisfying
\[0<\beta<2\delta:=2\rho(2c_0^2-1)\dist(x_0,\Omega).\]
Then, there exist constants $s_*>0$ and $C>0$ such that
\begin{align}\label{estimate'}&s^2\int_{Q_\tau}e^{2s\psi}|u|^2dxdt\\
&\le C\int_{Q_\tau} e^{2s\psi}|Au|^2dxdt+Cs\int_{\Sigma_\tau} e^{2s\psi}A\psi(H(t)\cdot \nu(x))|u|^2d\sigma dt\notag\\
&\quad\quad+Cs\int_\Omega e^{2s\psi(x,\tau)}|u(x,\tau)|^2dx\notag\end{align}
holds for all $s>s_*$ and $u\in H^1(Q_\tau)$. Here $d\sigma$ denotes the volume element of $\partial\Omega$.
\end{prop}

Note that the order of $s$ on the left-hand side of \eqref{estimate'} is different from the one on the left-hand side of \eqref{estimate}.

\begin{proof}[Proof of Proposition \ref{Carleman'}]
Set $z:=e^{s\psi}u$ and $Pz:=e^{s\psi}A(e^{-s\psi}z)$ for $u\in H^1(Q_\tau)$ and $s>0$. Since $\psi$ is smooth, it follows that $z\in H^1(Q_\tau)$. We note that
\[A\psi(x,t)=-\beta+2H(t)\cdot (x-x_0)\]
and
\[A^2\psi=2H'(t)\cdot (x-x_0)+2|H(t)|^2.\]
Since we have
\[Pz=Az-s(A\psi)z,\]
\begin{align}\label{inner_prod'}\|Pz\|_{L^2(Q_\tau)}^2&\ge s^2\|(A\psi)z\|_{L^2(Q_\tau)}^2+2(Az,-s(A\psi)z)_{L^2(Q_\tau)}\\
&=s^2\int_{Q_\tau}\Big(-\beta+2H(t)\cdot (x-x_0)\Big)^2|z|^2dxdt\notag\\
&\quad-2s\int_{Q_\tau}(\partial_tz+H(t)\cdot\nabla z)(A\psi)zdxdt\notag\\
&=s^2\int_{Q_\tau}\Big(-\beta+2H(t)\cdot (x-x_0)\Big)^2|z|^2dxdt\notag\\
&\quad-s\int_{Q_\tau}(A\psi)\partial_t(|z|^2)dxdt-s\int_{Q_\tau}(A\psi)H(t)\cdot\nabla(|z|^2)dxdt\notag\\
&=\int_{Q_\tau}\left[s^2\Big(-\beta+2H(t)\cdot (x-x_0)\Big)^2+s(A^2\psi)\right]|z|^2dxdt\notag\\
&\quad-s\int_{\Sigma_\tau} A\psi(H(t)\cdot\nu(x))|z|^2d\sigma dt-s\int_\Omega\Big[A\psi|z|^2\Big]_{t=0}^{t=\tau}dx\notag\end{align}
holds. For the fixed $x_0\in\overline{\Omega}^c$ so that $\displaystyle\min_{(x,t)\in\overline{Q_\tau}}\frac{H(t)\cdot(x-x_0)}{|H(t)||x-x_0|}\ge 2c_0^2-1(>0)$, it follows that
\begin{align*}H(t)\cdot(x-x_0)&=|H(t)||x-x_0|\frac{H(t)\cdot (x-x_0)}{|H(t)||x-x_0|}\\
&\ge\rho\dist(x_0,\Omega)\min_{(x,t)\in\overline{Q_\tau}}\frac{H(t)\cdot (x-x_0)}{|H(t)||x-x_0|}\\
&\ge\delta(>0)\end{align*}
for all $(x,t)\in \overline{Q_\tau}$ owing to \eqref{B.2} and \eqref{B.3}. We then obtain from \eqref{inner_prod'}
\begin{align*}\|Pz\|_{L^2(Q_\tau)}^2&\ge \int_{Q_\tau}\Big[(2\delta-\beta)^2s^2+O(s)\Big]|z|^2dxdt-s\int_{\Sigma_\tau} (A\psi)(H(t)\cdot\nu(x))|z|^2d\sigma dt\\
&\quad-s\int_\Omega A\psi(x,\tau)|z(x,\tau)|^2dx\end{align*}
as $s\to+\infty$. Hence, for all $0<\beta<2\delta$, there exist constants $s_*>0$ and $C>0$ such that
\begin{align*}s^2\int_{Q_\tau}e^{2s\psi}|u|^2dxdt&\le C\int_{Q_\tau}e^{2s\psi}|Au|^2dxdt+Cs\int_{\Sigma_\tau} e^{2s\psi}A\psi(H(t)\cdot \nu(x))|u|^2d\sigma dt\\
&\quad+Cs\int_\Omega e^{2s\psi(x,\tau)}|u(x,\tau)|^2dx\end{align*}
holds for all $s>s_*$ and $u\in H^1(Q_\tau)$.
\end{proof}

\subsection{Observability inequality for the non-degenerate case}
For the fixed $x_0\in\overline{\Omega}^c$ satisfying \eqref{angle'}, We define a positive number
\begin{equation}\label{T_0'} T_0':=\frac{\displaystyle\max_{x\in\overline{\Omega}}|x-x_0|^2-\min_{x\in\overline{\Omega}}|x-x_0|^2}{\delta},\end{equation}
where
\begin{equation}\label{delta'}\delta:=\rho(2c_0^2-1)\dist(x_0,\Omega)>0.\end{equation}

\begin{thm}
Let $T>0$, $c_0\in(\frac{1}{\sqrt{2}},1)$, $H\in C([0,T];\mathbb{R}^d)$, and $g\in L^2(\Sigma)$. Assume \eqref{positivity'} and $H\in C^1([0,T_1'];\mathbb{R}^d)$. If the positive number $t_1'>0$ defined by \eqref{def'} satisfies $T_0'<t_1'$ for some $x_0\in\overline{\Omega}^c$ satisfying \eqref{angle'}, then there exists a constant $C>0$ independent of $g\in L^2(\Sigma)$ such that for all $t\in[0,T]$,
\begin{equation}\label{observ_ineq'}\|u(\cdot,t)\|_{L^2(\Omega)}\le C\|g\|_{L^2(\Sigma)}\end{equation}
holds for all $u\in H^1(Q)$ satisfying \eqref{boundary}.
\end{thm}

\begin{proof}
Let $t_1'>0$ be the positive number defined by \eqref{def'} and $x_0\in\overline{\Omega}^c$ be the point satisfying \eqref{angle'} with
\[T_0'<t_1',\]
where $T_0'$ is defined by \eqref{T_0'}. Owing to Proposition \ref{continuity}, it suffices to prove \eqref{observ_ineq'} in the interval $[0,t_1']$. For the fixed $x_0\in\overline{\Omega}^c$, we take $0<\beta<2\delta$, where $\delta$ is defined by \eqref{delta'}, satisfying
\[(T_0'<)\frac{d_M-d_m}{\beta}<t_1',\]
where we recall
\[d_M:=\max_{x\in\overline{\Omega}}|x-x_0|^2,\quad d_m:=\min_{x\in\overline{\Omega}}|x-x_0|^2.\]
Then, there exists $\kappa>0$ such that
\begin{equation}\label{negative'}d_M-d_m-\beta t_1'<-\kappa.\end{equation}
Henceforth, by $C>0$ we denote a generic constant independent of $u$ and $g$ which may change from line to line, unless specified otherwise. We find that $H$ satisfies the assumptions \eqref{B.1}--\eqref{B.3} by taking $\tau=t_1'$ and $\theta_0=\frac{H(0)}{|H(0)|}$ needed for Proposition \ref{Carleman'}. Set $Q_{t_1'}:=\Omega\times(0,t_1')$ and $\Sigma_{t_1'}:=\partial\Omega\times(0,t_1')$. Applying Proposition \ref{Carleman'} to $u\in H^1(Q_{t_1'})$ satisfying \eqref{boundary} yields
\begin{align}\label{C'}s^2\int_{Q_{t_1'}} e^{2s\psi}|u|^2dxdt&\le Cs\int_{\Sigma_{t_1'}} e^{2s\psi}A\psi(H(t)\cdot \nu(x))|u|^2d\sigma dt\\
&\quad+Cs\int_\Omega e^{2s\varphi(x,t_1')}|u(x,t_1')|^2dx.\notag\end{align}
On the left-hand side of \eqref{C'}, we obtain
\begin{equation}\label{sharp'}s^2\int_{Q_{t_1'}}e^{2s\psi}|u|^2dxdt\ge s^2e^{2s(d_m-\beta\epsilon)}\int_0^\epsilon\int_\Omega |u|^2dxdt,\end{equation}
where $\epsilon\in(0,t_1')$ is an arbitrary small constant satisfying for all $x\in\overline{\Omega}$ and $0\le t\le\epsilon$,
\[\psi(x,t)> 0,\]
i.e.,
\begin{equation}\label{min'}d_m-\beta\epsilon> 0.\end{equation}
Furthermore, applying Lemma \ref{energy} in \eqref{sharp'}, we have
\begin{align}\label{lhs'}s^2\int_{Q_{t_1'}} e^{2s\psi}|u|^2dxdt&\ge s^2e^{2s(d_m-\beta\epsilon)}\int_0^\epsilon\int_\Omega |u|^2dxdt\\
&\ge \epsilon s^2e^{2s(d_m-\beta\epsilon)}\left(\|u(\cdot,0)\|_{L^2(\Omega)}^2-C\|g\|_{L^2(\Sigma)}^2\right).\notag\end{align}
Moreover, in regard to the second summand of the right-hand side of \eqref{C'}, applying Lemma \ref{energy} yields
\begin{align}\label{rhs'}&Cs\int_\Omega e^{2s\psi(x,t_1')}|u(x,t_1')|^2dx\\
&\quad\le Cse^{2s(d_M-\beta t_1')}\|u(\cdot,t_1')\|_{L^2(\Omega)}^2\notag\\
&\quad\le Cse^{2s(d_M-\beta t_1')}\left(\|u(\cdot,0)\|_{L^2(\Omega)}^2+C\|g\|_{L^2(\Sigma)}^2\right).\notag\end{align}
From \eqref{C'}, \eqref{lhs'}, and \eqref{rhs'}, we obtain
\begin{align*}&\epsilon s^2e^{2s(d_m-\beta\epsilon)}\left(\|u(\cdot,0)\|_{L^2(\Omega)}^2-C\|g\|_{L^2(\Sigma)}^2\right)\\
&\quad\le Cse^{2s(d_M-\beta t_1')}\left(\|u(\cdot,0)\|_{L^2(\Omega)}^2+C\|g\|_{L^2(\Sigma)}^2\right)+Cse^{Cs}\|g\|_{L^2(\Sigma)}^2,\end{align*}
i.e.,
\[e^{2s(d_m-\beta\epsilon)}\left(\epsilon s-Ce^{2s(d_M-d_m-\beta t_1'+\beta\epsilon)}\right)\|u(\cdot,0)\|_{L^2(\Omega)}^2\le Ce^{Cs}\|g\|_{L^2(\Sigma)}^2.\]
Applying \eqref{negative'} and \eqref{min'} to the left-hand side of the above inequality yields
\[\left(\epsilon s-Ce^{-2s(\kappa-\beta\epsilon)}\right)\|u(\cdot,0)\|_{L^2(\Omega)}^2\le Ce^{Cs}\|g\|_{L^2(\Sigma)}^2.\]
By choosing $s>s_*$ large enough to satisfy $\epsilon s-Ce^{-2s(\kappa-\beta\epsilon)}>0$ for the sufficiently small $\epsilon>0$ and applying Lemma \ref{energy} again on the left-hand side of the above inequality, we have
\[\|u(\cdot,t)\|_{L^2(\Omega)}^2\le C\|g\|_{L^2(\Sigma)}^2\]
for all $t\in[0,t_1']$.
\end{proof}

\begin{remark}
In the non-degenerate case, we focused only on the time interval $[0,t_1']$ near $0$ and proved the observability inequality under the assumption that $t_1'$ is large enough. Needless to say, if there exists a sufficiently long time interval $[t_*,t^*]\subset[0,T]$, if not near $0$, on which $\frac{H(t)}{|H(t)|}\cdot\frac{H(t_*)}{|H(t_*)|}\ge c_0$ holds, the observability inequality holds on the interval, which implies it holds also on $[0,T]$ by the similar arguments using Lemma \ref{energy} and Proposition \ref{continuity}.
\end{remark}

\section*{Acknowledgment}
This work was supported by Grant-in-Aid for Scientific Research (S) Grant Number JP15H05740, Grant-in-Aid for JSPS Fellows Grant Number JP20J11497, and Istituto Nazionale di Alta Matematica (IN$\delta$AM), through the GNAMPA Research Project 2020, titled ``Problemi inversi e di controllo per equazioni di evoluzione e loro applicazioni'', coordinated by the first author. Moreover, this research was performed in the framework of the French-German-Italian Laboratoire International Associ\'e (LIA), named COPDESC, on Applied Analysis, issued by CNRS, MPI, and IN$\delta$AM, during the IN$\delta$AM Intensive Period-2019, \lq\lq {\it Shape optimization, control and inverse problems for PDEs}'', held in Napoli in May-June-July 2019.

The authors thank Prof. Piermarco Cannarsa and Prof. Masahiro Yamamoto for the useful and interesting discussions about the topics of this paper.

\bibliographystyle{plain}
\bibliography{degenerate_JEE_r1}

\end{document}